\documentclass[12pt,a4paper]{amsart}
\usepackage[latin2]{inputenc}
\usepackage[mathscr]{eucal}
\usepackage{enumerate}

\usepackage{graphicx}
\usepackage{amsmath}
\usepackage{amsthm}
\usepackage{amssymb}
\usepackage{amsfonts}
\usepackage{latexsym}
\usepackage{color}

\newcommand{\comment}[1]{}

\newcommand{\Norm}[1]{{\left\|{#1}\right\|}}

\newcommand{\PP}{{\mathcal P}}
\newcommand{\Pn}{{\mathcal P}_n}
\newcommand{\PK}{{\mathcal P}_n(K)}
\newcommand{\ze}{\zeta}
\newcommand{\DK}{\partial K}

\newcommand{\rj}{r_j e^{i\varphi_j}}
\newcommand{\sj}{{\frac{\sin\varphi_j}{r_j}}}

\newcommand{\LL}{{\mathcal L}}

\newcommand{\KS}{K^{\star}}

\newcommand{\HH}{{\mathcal H}}

\newcommand{\HP}{{\mathbb H}}

\newcommand{\Z}{{\mathcal Z}}
\newcommand{\ZS}{{\mathcal Z}^{*}}
\newcommand{\W}{{\mathcal W}}
\newcommand{\RR}{{\mathbb R}}
\newcommand{\CC}{{\mathbb C}}
\newcommand{\ZZ}{{\mathbb Z}}
\newcommand{\NN}{{\mathbb N}}
\newcommand{\TT}{{\mathbb T}}
\newcommand{\DD}{{\mathbb D}}
\newcommand{\II}{{\mathbb I}}

\newcommand{\CR}{{\mathcal C}}
\newcommand{\SC}{{\mathcal S}}

\newcommand{\essinf}{ \mathop{\mathrm{ ess\,inf}}}
\newcommand{\intt}{ \mathop{\mathrm{int}}}
\newcommand{\con}{ \mathop{\mathrm{con}}}
\newcommand{\diam}{ \mathop{\mathrm{diam}}}
\newcommand{\width}{ \mathop{\mathrm{width}}}

\newtheorem{theorem}{Theorem}

\newtheorem{lemma}{Lemma}

\theoremstyle{claim}
\newtheorem{proposition}{Proposition}
\newtheorem{definition}{Definition}
\newtheorem{conjecture}{Conjecture}
\theoremstyle{definition}

\newcommand{\de}{\delta}

\newcommand{\tht}{\theta}

\newcounter{othm}
\def\theothm{\Alph{othm}}
\newenvironment{othm}{
  \em
  \vskip 0.10in
  \refstepcounter{othm}
  \noindent{\bf Theorem\ \theothm}
}{\vskip 0.10in}

\newenvironment{olemma}{
  \em
  \vskip 0.10in
  \refstepcounter{othm}
  \noindent{\bf Lemma\ \theothm}
}{\vskip 0.10in}

\newcounter{rev}
\newcounter{rep}
\begin{document}

\title[$L^q$ Tur\'an inequalities on Er\H od domains]{Tur\'an type converse Markov inequalities in $L^q$ \\ on a generalized Er\H od class of convex domains}

\author{Polina Yu. Glazyrina, Szil\'ard Gy. R\'ev\'esz}

\address
{Polina Yu. Glazyrina \newline  \indent Institute of Mathematics and Computer Sciences, Ural Federal University, \newline  \indent Ekaterinburg, RUSSIA \newline  \indent  and \newline  \indent
Institute of Mathematics and Mechanics, \newline  \indent Ural Branch of the Russian Academy of Sciences, \newline  \indent Ekaterinburg, RUSSIA}
\email{polina.glazyrina@urfu.ru}

\address{Szil\'ard Gy. R\'ev\'esz \newline  \indent Institute of Mathematics, Faculty of Sciences \newline \indent Budapest University of Technology and Economics \newline  \indent Budapest, M\H uegyetem rkp. 3-9. \newline \indent 1111 HUNGARY \newline  \indent and \newline  \indent A. R\'enyi Institute of Mathematics \newline \indent Hungarian Academy of Sciences, \newline \indent Budapest, Re\'Altanoda utca 13-15. \newline \indent 1053 HUNGARY} \email{revesz@renyi.mta.hu}

\date{\today}

\maketitle

\let\oldfootnote\thefootnote
\def\thefootnote{}
\footnotetext{} \footnotetext{This work was supported by the Program for State Support  of Leading Scientific Schools  of the Russian Federation (project no. NSh-9356.2016.1) and by the Competitiveness Enhancement Program of the Ural Federal University (Enactment of the Government of the Russian Federation no. 211 of March 16, 2013, agreement no. 02.A03.21.0006 of August 27, 2013)
and by Hungarian Science Foundation Grant \#'s 
 K-119528, NK-104183, K-109789.}
\let\thefootnote\oldfootnote

\bigskip
\bigskip



\begin{abstract}
P. Tur\'an was the first to derive lower estimations on the  uniform norm of the derivatives of polynomials $p$ of uniform norm $1$ on the interval $\II:=[-1,1]$  and the disk $\DD:=\{z\in\CC~:~|z|\le 1\}$, under the normalization condition that the zeroes of the polynomial $p$ in question all lie in $\II$ or $\DD$, resp. Namely, in 1939 he proved that with $n:=\deg p$ tending to infinity, the precise growth order of the minimal possible derivative norm is $\sqrt{n}$ for $\II$ and $n$ for $\DD$.

Already the same year J. Er\H od considered the problem on other domains.
In his most general formulation, he extended Tur\'an's order $n$ result on $\DD$ to a certain general class of piecewise smooth convex domains. Finally, a decade ago the growth order of the minimal possible norm of the derivative was proved to be $n$ for all compact convex domains.

Tur\'an himself gave comments about the above oscillation question in $L^q$ norm on $\DD$. Nevertheless, till recently results were known only for $\II$, $\DD$ and so-called $R$-circular domains.
Continuing our recent work, also here we investigate the Tur\'an-Er\H od problem on general classes of domains.
\end{abstract}


\tableofcontents



\section{Introduction}\label{sec:intro}

Denote by $K\subset \CC$ a compact subset of the complex plane,
with the most notable particular cases being the unit disk
$\DD:=\{ z \in \CC ~:~ |z|\le 1\}$ and the unit interval
$\II:=[-1,1]$.

As a kind of converse to the classical inequalities of
Bernstein \cite{Bernstein, Bernstein_com, Riesz} and Markov
\cite{Markov} on the upper estimation of the norm of the
derivative of polynomials, in 1939 Paul Tur\'an \cite{Tur}
started to study converse inequalities of the form
$$\Norm{p'}_K\ge c_K n^A \Norm{p}_K.
$$
 Clearly such a converse
can only hold if further restrictions are imposed on the
occurring polynomials $p$. Tur\'an assumed that all zeroes of
the polynomials belong to $K$. So denote the set of complex
(algebraic) polynomials of degree (exactly) $n$ as $\PP_n$, and
the subset with all the $n$ (complex) roots in some set
$K\subset\CC$ by $\PK$.

Denote by $\Gamma$ the boundary of $K$. The (normalized)
quantity under our study in the present paper is the ``inverse
Markov factor" or ``oscillation factor"
\begin{equation}\label{Mdef}
M_{n,{q}}(K):=\inf_{p\in \PK} \frac{\Norm{p'}_{L^q(\Gamma)}}{\Norm{p}_{L^q(\Gamma)}},
\end{equation}
where, as usual,
\begin{equation}
\label{Oscillationdef}
\begin{aligned}
\Norm{p}_{q}:&=\Norm{p}_{L^q(\Gamma)}:=\left(\int_{\Gamma} |p(z)|^q|dz|\right)^{1/q},
\quad (0<q<\infty) \notag
\\ \Norm{p}_K:=\|p\|_\infty:&=\Norm{p}_{L^\infty(\Gamma)}:=\sup_{z\in \Gamma} |p(z)|=\sup_{z\in K}|p(z)|.
\end{aligned}
\end{equation}

We are discussing Tur\'an-type inequalities on \eqref{Mdef} for
general convex sets, so some geometric parameters of the convex domain $K$ are involved naturally.
We write $d_K:=\diam K$ for the {\em diameter} of $K$, and $w_K:={\width}~ K$ for the {\em minimal width} of $K$. That is,
\begin{equation*}\label{diameterdef}
d_K:= \sup_{z', z''\in K} |z'-z''|, \qquad
w_K:= \inf_{\gamma\in [-\pi,\pi]} \left( \sup_{z\in K} \Re
(ze^{i\gamma}) - \inf_{z\in K} \Re (ze^{i\gamma}) \right).
\end{equation*}
Note that a (closed) convex \emph{domain} is a (closed),
bounded, convex set $K\subset\CC$ \emph{with nonempty
interior}, hence $0<w_K\le d_K<\infty$.
We also use the notation $\Delta_K$ for the \emph{transfinite diameter} of $K$ (see, e.g., \cite[5.5]{Rans}).

A detailed account of the results concerning Tur\'an-type inequalities for general convex sets is given in \cite{PR}, so here let us confine ourselves only to a less exhaustive history of the 
topic. In 1939, Tur\'an \cite{Tur} proved the following.
\begin{othm}{\bf(Tur\'an).}\label{oth:Turandisk}
If $p\in \PP_n(\DD)$, then we
have
\begin{equation}\label{Turandisk}
\Norm{p'}_\DD\ge \frac n2 \Norm{p}_\DD~.
\end{equation}
If $p\in\PP_n(\II)$, then we have
\begin{equation}\label{Turanint}
\Norm{p'}_\II\ge \frac {\sqrt{n}}{6} \Norm{p}_\II~.
\end{equation}
\end{othm}

Inequality \eqref{Turandisk} of Theorem \ref{oth:Turandisk} is
best possible. Regarding \eqref{Turanint}, Tur\'an pointed out
by example of $(1-x^2)^{n}$ that the $\sqrt{n}$ order cannot be
improved upon, even if the constant is not sharp. The precise value of the constants and
the extremal polynomials were computed for all fixed $n$ by
Er\H{o}d in \cite{Er}.

The key to get \eqref{Turandisk} is the following straightforward
observation.
\begin{olemma}{\bf(Tur\'an).}\label{Tlemma} Assume that $z\in\partial K$ and
that there exists a disc $D_R=\{\zeta\in \CC~:~ |\ze-z_0|\le R\}$ of radius $R$ so that $z\in\partial D_R$ and $K\subset D_R$. Then for all $n\in \NN$ and $p\in\PK$ we have
\begin{equation}\label{Rdisc}
|p'(z)| \ge \frac n{2R} |p(z)|.
\end{equation}
\end{olemma}

For the easy and direct proof see any of the references \cite{Tur, LP, Rev3, SofiaCAA, PR}. Levenberg and Poletsky \cite{LP} found it worthwhile to
formally define the crucial property of convex sets, used here.

\begin{definition}[{\bf Levenberg-Poletsky}]\label{def:Rcircular}
A set $K\subset \CC$ is called \emph{$R$-circular}, if for any  $z\in\partial K$ there exists a disk $D_R$ of radius $R$, such that
$z\in\partial D_R$ and $K\subset D_R$ .
\end{definition}

Thus in particular for any $R$-circular $K$ and $p\in \PP_n(K)$ at the
boundary point $z\in\partial K$ with $\|p\|_K=|p(z)|$ we can
draw the disk $D_R$ and get $\|p'\|_K \ge \dfrac{1}{2R} n \|p\|_K$, equivalently to \eqref{Turandisk}.

Er\H od continued the work of Tur\'an already the same year, investigating the inverse Markov factors of domains with some favorable geometric properties. The most general domains with $M_{n,\infty}(K)\gg n$, found by Er\H od, were described on p. 77 of \cite{Er}.

\begin{othm}{\bf(Er\H od).}\label{oth:transfquarter} Let $K$ be any convex domain bounded by finitely many Jordan arcs,
joining at vertices with angles $<\pi$, with all the arcs being $C^2$-smooth and being either straight lines of length $<\Delta_K$ or having positive curvature bounded away from $0$ by a fixed positive constant $\kappa>0$.

Then there is a constant $c(K)$, such that $M_{n,\infty}(K)\geq c(K) \,n$ for all $n\in\NN$.
\end{othm}

As is discussed in \cite{PR}, this result covers the case of regular $k$-gons $G_k$ for $k\ge 7$, but not the square $Q=G_4$. As for that matter, Erd\'elyi proved that $G_4$ too has order $n$ oscillation, however, this advance appeared only much later in \cite{E}.

A lower estimate of the inverse Markov factor for all compact convex sets (and of the same $\sqrt{n}$ order as was known for the interval) was obtained in full generality only in 2002 by Levenberg and Poletsky \cite[Theorem 3.2]{LP}.

Since $\sqrt{n}$ was already known to be the right order of growth for the inverse Markov factor of the interval $\II$, it remained to clarify the right order of oscillation for compact convex \emph{domains} with nonempty interior. This was solved a decade ago in \cite{Rev2};
for the fact that it is indeed the precise order 
see \cite{Rev3, PR, SofiaCAA}.
\begin{othm}{\bf(Hal\'{a}sz-R\'ev\'esz).}\label{th:convexdomain}
Let $K\subset \CC$ be any compact convex domain. Then for all  $p\in
\PK$ we have
\begin{equation*}\label{genrootineq}
\Norm{p'}_K\ge 0.0003 \frac{w_K}{d_K^2} n  \Norm{p}_K~.
\end{equation*}
\end{othm}

\bigskip
There are many papers dealing with the $L^q$-versions of Tur\'an's inequality.
The estimation of the $L^q$ norm, or of any 
weighted $L^q$ norms,
goes the same way if we have a pointwise estimation for all, (or for linearly almost all), boundary points.
Already Tur\'an himself mentioned in \cite{Tur} that by~\eqref{Rdisc} for any $q>0$ we have
$
M_{n,q}(\DD)\ge n/2.
$
Levenberg and Poletsky extended this observation to $R$-circular domains in \cite{LP}.

\begin{othm}{\bf (Levenberg-Poletsky).}\label{th:Levenberg-Poletsky} Assume that the convex, compact domain $K$ is $R$-circular with a certain radius $0<R<\infty$.
Then at any boundary point $z \in \partial K$ we have $|p'(z)| \ge \dfrac{n}{2R}  |p(z)|$.
Consequently, for any weighted $L^q$ norm $\|\cdot\|$, we have $\| p'\| \geq \dfrac{n}{2R} \|p\| ~( \forall p\in \PK)$. In particular, $M_{n,q}(K) \geq \dfrac{n}{2R}$.
\end{othm}

These estimates are not necessarily optimal, though. For more details about special results on $\DD$ and $\II$ see the detailed account of \cite{PR} and the references therein.

\bigskip

In case we discuss maximum norms, one can assume that $|p(z)|$ is maximal,
and it suffices to obtain a lower estimation of $|p'(z)|$ only at such a special point -- for general norms, however, this is not sufficient.
The above results work only for we have a pointwise inequality of the same strength \emph{everywhere}, or almost everywhere.

The situation becomes considerably more difficult, when such a statement cannot be proved. E.g. if the domain in question is not strictly convex (there is a line segment on the boundary), then the zeroes of the polynomial can be arranged so that even some zeroes of the derivative lie on the boundary, and at such points $p'(z)$ -- even $p'(z)/p(z)$ -- can vanish. As a result, at such points no fixed lower estimation can be guaranteed, and lacking a uniformly valid pointwise comparision of $p'$ and $p$, a direct conclusion cannot be drawn either.

This explains why the cases of the interval $\II$ and non strictly convex domains are much more complicated for the integral mean norms. Nevertheless, in a series of papers \cite{Zhou84, Zhou86, Zhou92, Zhou93, Zhou95}, Zhou proved that for the interval $\II$
$
M_{n,q}(\II)\ge c_{q}\sqrt{n}.
$
The best possible constants in some cases were found by Babenko and  Pichugov ~ \cite{BabenkoU86}, Bojanov ~\cite{Bojanov93} and Varma~\cite{Varma88_83}.

The classical inequalities of Bernstein and Markov are
generalized
to various differential operators, too, see  \cite{de-Bruijn, Rahman1969, Arestov, Rather,  AzizRather2007, Dewan, Jain1997, Jain2000}.
In this context, also Tur\'an type converses have been already investigated,  see e.g. \cite{Akopyan00, Dewan, Jain1997, Jain2000}.

\bigskip

Recently, we obtained some order $n$ oscillation results for certain convex domains without any $R$-circularity condition or strict convexity. To formulate this, let us
first introduce another geometrical notion, namely,
the \emph{depth} of a convex domain $K$ as
\begin{equation*}\label{eq:bodydepth}
h_K:=\sup \{ h\ge 0 ~:~ \forall \zeta\in\partial K ~\exists ~ {\rm a}~{\rm normal} ~ {\rm line} ~ \ell ~ {\rm at} ~ \zeta ~ {\rm to} ~ K ~ {\rm with} ~~ |\ell\cap K| \ge h \}.
\end{equation*}

We say that the convex domain $K$ has {\em fixed depth} or {\em positive depth}, if $h_K>0$.
Although this is quite a general class, which e.g. contains all smooth compact convex domains, observe that the regular triangle has $h_K=0$, as well as any polygon having some acute angle. For more about this class see \cite{PR}, where the following was proved.
\begin{othm}{\bf.}\label{th:posdepth} Assume that $K\subset \CC$ is a convex domain with positive depth $h_K>0$. Then for any $1\le q <\infty$, any $n\in \NN$ and any $p\in\PK$ it holds
\begin{equation*}\label{eq:ordernLq}
\|p'\|_{q,K} \ge c_K n \|p\|_{q,K} \qquad \left( c_K:=\frac{h_K^4}{3000 d_K^5} \right).
\end{equation*}
\end{othm}
From the other direction, we also proved that one cannot expect more than order $n$ growth of $M_{n,q}(K)$. In fact, in this direction our result was more general, but here we recall only a combination of Theorem 5 and Remark 6 of \cite{PR}.
\begin{othm}{\bf.}\label{th:orderupper} Let $K\subset \CC$ be any compact, convex domain.
Then for any $q \ge 1$ and any $n\in \NN$ there exists a polynomial $p\in \PK$  satisfying $\|p'\|_{L^q(\partial K)} < \dfrac{15}{d_K} n \|p\|_{L^q(\partial K)} $.
\end{othm}
Based on these findings, 
we concluded in \cite{PR} with the following conjecture.
\begin{conjecture} For all compact convex domains $K\subset \CC$ there exists $c_K>0$ such that for any $p\in\PK$ we have $\|p'\|_{L^q(\DK)} \ge c_K n \|p\|_{L^q(\DK)}$. That is, for any compact convex domain $K$ the growth order of $M_{n,q}(K)$ is precisely $n$.
\end{conjecture}

The aim of the present work is to prove the validity of the above Conjecture for another class of compact convex domains, containing the class of Er\H od in Theorem \ref{oth:transfquarter}. One point is that in this class there are domains having even some acute angles at certain boundary points (and thus having zero depth). So the result essentially adds to the families of known classes having the property that $M_{n,q}(K)\asymp n$.

\section{Formulation of the result for generalized Er\H od type domains}\label{sec:Etype}

Before formulating our result, we need to introduce some geometrical notations.

We start with a \emph{convex, compact domain} $K\subset \CC$. Then its interior $\intt K \ne \emptyset$ and $K=\overline{\intt K}$, while its boundary $\Gamma:=\partial K$ is a convex Jordan curve. More precisely, $\Gamma = {\mathcal R}(\gamma)$ is the \emph{range} of a continuous,  convex, closed Jordan curve $\gamma$ on the complex plane $\CC$.
As the curve $\gamma$ is  convex, it  has finite arc length $L$ and we will restrict ourselves to \emph{parametrization with respect to arc length}.

If the parameter interval of the Jordan curve $\gamma$ is $[0,L]$, then this means, that $\gamma: [0,L] \to \CC$ is continuous, convex, and one-to-one on $[0,L)$, while $\gamma(L)=\gamma(0)$. While this compact interval parametrization is the most used setup for curves, we need an essentially equivalent interpretation with this, too: the periodically extended interpretation  $\gamma: \RR \to \partial K$  with $\gamma(t):=\gamma(t-[t/L]L)$ defined periodically all over $\RR$.

The parametrization $\gamma: \RR \to \partial K$ defines a unique ordering of points, which we assume to be positive in the counterclockwise direction, as usual.
When considered locally, i.e. with parameters not extending over a set longer than the period, this can be interpreted as ordering of the image (boundary) points themselves: we always implicitly assume, that a proper cut of the torus $\RR/L\ZZ$ is applied at a point to where the consideration is not extended, and then for the part of boundary we consider, the parametrization is one-to-one carrying over the ordering of the cut torus to $\partial K$.

Arc length parametrization has an immediate consequence also regarding the derivative, which must then have $|\dot{\gamma}|=1$, whenever it exists, i.e. (linearly) a.e. on $\RR$. Since $\dot{\gamma} :\RR \to \partial \DD$, we can as well describe the value by its angle or argument: the derivative angle function will be denoted by $\alpha:=\arg \dot{\gamma} : \RR \to \RR$.
Since, however, the argument cannot be defined on the unit circle without a jump,  we decide to fix one value and then define the extension continuously: this way $\alpha$ will not be periodic, but we will have rotational angles depending on the number of (positive or negative) revolutions, if started from the given point.
With this interpretation, $\alpha$ is an a.e. defined nondecreasing real function with $\alpha(t)-\frac{2\pi}{L} t$ periodic (by $L$) and bounded.
Angular values attained by $\alpha(t)$ are then ordered the same way as boundary points and parameters.
In particular, for a subset not extending to a full revolution,
the angular values are uniquely attached to the boundary points and parameter values and they can be similarly 
ordered
considering a proper cut.

Let $\alpha_{-}$ and $\alpha_{+}$ be 
the left- resp. right-continuous extensions of $\alpha$.
The geometrical meaning is that if for a parameter value $\tau$ the corresponding boundary point is $\gamma(\tau)=\ze$, then $[\alpha_{-}(\tau),\alpha_{+}(\tau)]$ is precisely the interval of values $\beta \in \TT$ such that the straight lines $\{\zeta+e^{i\beta}s~:~ s\in \RR\}$ are supporting lines to $K$ at $\zeta \in \partial K$.
We will interpret $\alpha$ as a multi-valued function, assuming all the values in $[\alpha_{-}(\tau),\alpha_{+}(\tau)]$ at the point $\tau$.

The curve $\gamma$ is differentiable at $\zeta=\gamma(\theta)$ if and only if $\alpha_{-}(\theta)=\alpha_{+}(\theta)$; in this case the unique tangent of $\gamma$ at $\zeta$ is $\zeta+e^{i\alpha}\RR$ with $\alpha=\alpha_{-}(\theta)=\alpha_{+}(\theta)$. Also note that by convexity the curvature $\ddot{\gamma}\ge 0$ exists and is nonnegative linearly a.e.

For obvious geometric reasons we call the jump function  $\Omega:=\alpha_{+}-\alpha_{-}$ the {\em supplementary angle} function. This is zero except for a countable set, and has positive values such that the total sum of the (possibly infinite number of) jumps on $[0,L]$ does not exceed the total variation of $\alpha$ on $[0,L]$, i.e. $2\pi$.

\begin{definition}\label{def:Edomain} We say that a compact convex domain $K\subset \CC$ is an $E$-\emph{domain} --- more precisely,
it is an $E(k,d,\Delta,\kappa,\xi,\delta)$-domain with the positive parameters $k\in \NN$, $d,\Delta,\kappa, \xi, \de >0$
satisfying $0<\Delta\le d/2,$ $0<\delta\le\Delta/2,$ $0<\xi\le \pi/2$ --- if the following properties hold.
\begin{enumerate}
\item $d=d_K$ and $\Delta=\Delta_K$ are the diameter and transfinite diameter of $K$, resp.
\item The boundary curve $\gamma:[0,L] \to \Gamma=\DK$ can be decomposed to a finite number $k$ of adjoining pieces $\Gamma=\cup_{j=1}^k \Gamma_j$, each $\Gamma_j$ being (the range of) a simple, convex Jordan arc $\gamma_j: I_j:=[v_j,v_{j+1}]\to \Gamma_j \subset \DK$ (with $j=1,\dots,k$ and $v_{k+1}:=v_1$);
\item Each Jordan arc $\gamma_j$ ($j=1,\dots,k$) belongs to one of the categories below. \begin{enumerate}[(i)]
    \item It either satisfies $|\ddot{\gamma}_j|\ge
    \kappa$ a.e. on $I_j$ (when we call it a \emph{curved} piece of the boundary);
    \item or it is a straight line segment of length $L_j:=|\Gamma_j|= v_{j+1}-v_j \le \Delta - \de$ (in which case it is called a \emph{straight} piece of the boundary).
    \end{enumerate}
\item At each $V_j=\gamma(v_j)$ ($j=1,\dots,k$), the boundary curve has an outer angle
 $\Omega(V_j):=\alpha_{+}(v_j)-\alpha_{-}(v_j) \ge \lambda(j) \xi$, where $\lambda(j)$ is the number of \emph{straight}
boundary pieces among $\Gamma_{j-1}$ and $\Gamma_{j}$, joining at $V_j$.
\end{enumerate}
\end{definition}
That is, we assume that there is a decomposition of the boundary, with each piece classified as either a ``straight'' or a ``curved" arc,
 and the parameters\footnote{By assumption, we took $\xi:=\min(\xi',\pi/2)$ and $\de:=\min(\de',\Delta/2)$.}
$$
\kappa:=\min\limits_{j \colon \Gamma_j  \atop \textrm{ is curved}}\essinf |\ddot{\gamma_j}|, \quad
   \xi':=\min\limits_{j \colon \Gamma_{j-1} \textrm{ or } \Gamma_{j} \atop \textrm{ is straight}} \Omega(V_j)/\lambda(j),\quad
\delta':=\min\limits_{j \colon \Gamma_j \atop \textrm{ is straight}} (\Delta-L_j)
$$
are positive. Clearly, this class contains that in Theorem \ref{oth:transfquarter}, some additional generality lying in the facts that we drop any smoothness condition on the curved arcs $\Gamma_j$ and require the separation of the curvature from zero only almost everywhere. As for that matter, this also allows \emph{joining into one piece} any consecutive curved arcs (which was not possible under the original $C^2$ condition of Er\H od). Also, strictly speaking $\DD$ (or other $C^2$ domains with curvature otherwise exceeding a constant $\kappa>0$) do not belong to the original class of Er\H od, for the parametrization a somewhat artificial vertex point $V_1$ have to be taken, where, however, no jump of the angle exhibits itself. This explains why it is advantageous to involve here also the quantity $\lambda(j)$: at joining points of only curved pieces no jump need to be assumed.

Note that assuming $|\ddot{\gamma}_j|\ge \kappa$ only a.e. is the relatively recent generality, achieved in the topic of Blaschke Rolling Ball Theorems. The classical result of Blaschke \cite{Bla} gives only that a convex domain $K$ with a $C^2$-smooth boundary curve $\gamma$, having curvature $|\ddot{\gamma}| \ge \kappa_0>0$ \emph{everywhere} along the boundary, 
is $1/\kappa_0$-circular. This has been generalized recently even to a.e. conditions, as we will discuss somewhat below. Although these theorems cannot be directly used here, due to the presence of some straight line pieces on the boundary of our Er\H od type domains, in the proofs we will still invoke them and therefore achieve this greater generality, too.

So we will dedicate the rest of the paper to the proof of the following.

\begin{theorem}\label{th:E} Let $K\subset \CC$ be an $E(k,d,\Delta,\kappa,\xi,\delta)$-domain as defined above. Then there exists a constant $c=c_K$ (depending explicitly on the parameters $k,d,\Delta,\kappa,\xi,\delta$) such that for any $q\ge 1$, any $n\in \NN$ and any $p\in\PK$ we have
$$
\|p'\|_q \ge c_K n \|p\|_q.
$$
\end{theorem}

\section{Technical preparations for the investigation of~$L^q(\partial K)$ norms}\label{s:Lemmas}

\begin{lemma}\label{l:Nikolskii} For any polynomial $p$ of degree at most $n$ and any $q>0$ we have that
\begin{equation}\label{eq:Nikolskii}
\|p\|_{L^q(\DK)} \ge \left( \frac{d}{2(q+1)}\right)^{1/q}
~\|p\|_{L^\infty(\DK)} ~ n^{-2/q} .
\end{equation}
\end{lemma}
For a proof of this Nikolskii-type estimate, see \cite[Lemma 1]{PR}.

Next, let us define the subset $\HH:=\HH_K^q(p) \subset \DK$ the following way.
\begin{equation}\label{eq:Hsetdef}
\HH:=\HH_K^q(p):=\{\zeta\in \partial K~: ~ |p(\zeta)| > c n^{-2/q} \|p\|_\infty\}
\quad \left( c:=\left(\frac{1}{8\pi(q+1)}\right)^{-1/q} \right).
\end{equation}
Then in \cite[Section 3.1]{PR} it was deduced from the above Lemma \ref{l:Nikolskii} that we have
\begin{lemma}\label{l:Hlogp} Let $\HH \subset \DK$ be defined according to \eqref{eq:Hsetdef} and let $q>0$ be arbitrary. Then for all $n\in \NN$ and for all $p \in \PP_n$ we have
\begin{equation*}\label{eq:pqintegralonH}
\int_{\HH} |p|^q \geq \frac12 \|p\|^q_{L^q(\DK)}.
\end{equation*}
Furthermore, for any point $\ze \in \HH$, and for any $n\in \NN$ and any $p\in \PK$ we also have
\begin{equation*}\label{eq:pnormperponH}
\log \frac{\|p\|_\infty}{|p(\ze)|} \le \log (16\pi) + 2 \log n  \quad \left( \le 4 \log n ~{\rm if} ~ n\ge 8 \right).
\end{equation*}
\end{lemma}

Note that the proof of Theorem \ref{oth:transfquarter} by Er\H od in \cite{Er} was slightly incomplete\footnote{For a detailed analysis and the necessary slight addition to the argument, see \cite{Er} and \cite{Rev3}.} and went along somewhat different lines, basically trying to follow the geometrical features by a direct calculation, which otherwise could have been treated--as we will do here--by the Blaschke Rolling Ball Theorem (seemingly unknown to Er\H od). While we are utilizing the essential ideas of the method of Er\H od, here we make explicit use of these Rolling Ball Theorems, and, in fact, capitalize on the far-reaching generalizations known by now in geometry. More precisely, the key to treat the curved arcs of the boundary of $K$ will be the next lemma.
\begin{olemma}{\bf(Strantzen).}\label{l:roughcurvature} Let the compact convex domain $K$ have boundary curve $\Gamma=\partial K$ and let $\kappa>0$ be a fixed constant. Assume that the convex boundary curve $\Gamma$ (which is, by convexity, twice differentiable linearly almost everywhere) satisfies the curvature condition $|\ddot{\Gamma}|\geq \kappa$ almost everywhere. Then to each boundary point $\zeta\in\partial K$ there exists a disk $D_R$ of
radius $R=1/\kappa$, such that $\zeta\in\partial D_R$, and
$K\subset D_R$. That is, $K$ is $R=1/\kappa$-circular.
\end{olemma}
\begin{proof} This result is essentially the far-reaching, relatively
recent generalization of Blaschke's Rolling Ball Theorem by
Strantzen. A reference for it is Lemma 9.11 on p. 83 of \cite{BS}.
For more details on this, as well as for some new approaches to the proof of
this generalization of the classical Blaschke Rolling Ball
Theorem, see \cite{Rev4}.
\end{proof}


From here it is easy to see the following result, which is the $k=1$ boundary curve case of Theorem \ref{th:E}, when that one boundary piece is necessarily curved (since otherwise we encounter the degenerate case of an interval only).
\begin{proposition}\label{th:Lqcircular} Assume that the boundary curve $\gamma:[0,L]\to\Gamma:=\partial K$ of the convex domain $K$ satisfies at (linearly) almost all points the condition that it has a curvature, not smaller than a given positive constant $\kappa$, i.e. $|\ddot{\gamma}| \geq \kappa \, (>0)$ a.e. Then we have for any $n\in \NN$ and any $p\in\PK$ that $|p'(z)|\ge \dfrac{\kappa}{2}n|p(z)|$ ($z\in \partial K$) and for any weighted $L^q$ norm $\|\cdot\|$ on $\Gamma=\partial K$, we have $\| p'\| \geq \dfrac{\kappa}{2} n \|p\|$. In particular, $M_{n,q}(K) \geq \dfrac{\kappa}{2} n$.
\end{proposition}

\begin{proof} This was implicitly contained already in \cite{Rev3, SofiaCAA} and was explicitly formulated as Theorem M in \cite{PR}. The proof is clear: the geometric condition entails the $R$-circularity of the domain with $R:=1/\kappa$ in view of Lemma \ref{l:roughcurvature}, whence Theorem \ref{th:Levenberg-Poletsky} furnishes the result.
\end{proof}


The other key and innovative feature of the original work of Er\H od was invoking Chebyshev's Lemma, which we will use in the slightly more general form of an estimation using the transfinite diameter.
\begin{olemma}{\bf (Transfinite Diameter Lemma).}\label{l:FeketeSzego} Let $K\subset \CC$
be any compact set, $n\in \NN$ arbitrary, and $p\in\Pn$ be a monic polynomial, i.e. assume that $p(z)=\prod_{j=1}^n (z-z_j)$.
Then we have $\|p\|_{L^\infty(\partial K)} \ge \left(\Delta_K\right)^n$.
\end{olemma}
\begin{proof} In various forms essentially this was first proved by Fekete, Faber and Szeg\H o. For details and references see \cite[Lemma P]{PR} and its discussion there.
\end{proof}

We will also use a classical result of Gabriel \cite[Theorem 5.1]{Gabriel}, see also \cite{PR}. In fact, we will need the following consequence of Gabriel's Lemma.
\begin{lemma}\label{l:Gabriel} If $n\in \NN$ and if $p \in \PK$, then for any $q \ge 1$ it holds
$$
\|p'\|_q \ge  0.022 \dfrac1{d} \|p\|_q.
$$
\end{lemma}
\begin{proof} See Lemma 3 of \cite{PR}, where this is derived from the classical result of Gabriel.
\end{proof}

\section{An $R$-circularity argument on the curved pieces of the boundary}\label{s:curvedpart}

We will use Tur\'an's pointwise estimate on the curved arcs. For uniformity of that argument, we need the next 
geometrical lemma.
\begin{lemma}\label{l:partialRcircular} If $K$ is an $E(k,d,\Delta,\xi,\kappa,\de)$-domain, then there exists some $R:=R_K:=R(\Delta,\kappa,\xi)<\infty$ such that $K$ is \emph{partially $R$-circular} in the sense that to all \emph{curved} Jordan arcs $\Gamma_j$
in the decomposition (2) of $\Gamma=\DK$, and to all points $z\in \Gamma_j$, there exists a disk $D_R$ of radius $R$ such that $z\in\partial D_R$ and $K\subset D_R$.
Moreover, one can take $R_K:=\max\left\{1/\kappa, \ \Delta/(2\sin \xi) \right\}$.
\end{lemma}

\begin{proof}
The key to our proof is an application of the Strantzen result Lemma \ref{l:roughcurvature}, however, not directly to $K$, but to another domain $\KS$.

For this we will replace every \emph{straight} line segment parts $\Gamma_\ell$ of $\Gamma$ by a circular arc $\Lambda_\ell$ and thus obtain a new curve $C$ having a curvature exceeding some fixed positive number linearly almost everywhere along $C$. Moreover, we will do this in such a way that the domain $\KS$, encircled by $C$, will still remain convex, and it will contain $K$, so that in particular the disks, constructed using Lemma \ref{l:roughcurvature} for any boundary point $z \in \Gamma_j \subset C$ for any \emph{curved} arc $\Gamma_j$ of the boundary $\Gamma=\DK$, will also cover $K$ together with $\KS$.

So let us number the straight line segments as $\Gamma_{j_\nu},$ $\nu=1,\ldots,m$ (where $m\le k$ is the number of straight line pieces of the decomposition (2) of the boundary) with $j_\nu < j_{\nu+1},$ $\nu=1,\ldots,m-1$. All these will be replaced by a circular arc $\Lambda_{j_\nu}$, and so the new curve will be $C:=\cup_{j=1}^k C_j$, where $C_j:=\Gamma_j$ for the original curved arcs, and $C_\ell:=\Lambda_\ell$ if $\ell=j_\nu$ for some $\nu=1,\dots,m$, i.e. if $\Gamma_\ell$ was a straight line segment piece of the boundary.

For a completely definite construction, it remains to define the circular arcs $\Lambda_\ell$. So let now $\ell=j_\nu$ for some $1\le \nu \le m$. The circular arc $\Lambda_\ell$ will join the vertices $V_\ell$ and $V_{\ell+1}$ (the same way as $\Gamma_\ell$ did), and will run in that halfplane of the two ones defined by the straight line passing through $V_\ell$ and $V_{\ell+1}$, which is free from the interior points of $K$. This ensures that the curve $C$ will encircle $\KS=K\cup \left(\cup_{\nu=1}^m S_{j_\nu} \right)$, where $S_{j_\nu}$ are the disk caps, lying fully on the other side of $\Gamma_{j_\nu}$ than $K$, between $\Gamma_{j_\nu}$ and $\Lambda_{j_\nu}$. Finally, the radius of the circular arc $\Lambda_\ell$ will be chosen as
\begin{equation}\label{Rell}
R_\ell := \frac{|V_{\ell+1}-V_{\ell} |}{2\sin \xi } \le \frac{\Delta}{2\sin \xi}.
\end{equation}
As $0<\sin \xi \le 1$, the division results in a well-defined finite quantity $R_\ell$ exceeding or equal to the half of the length $L_\ell=|\Gamma_\ell|=|V_{\ell+1}-V_\ell|$, and so the circular arc $\Lambda_\ell$ is unambiguously defined. Moreover, by the sine rule the angle between the segment $\Gamma_\ell$ and the circular arc $\Lambda_\ell$ will be exactly $\xi$.

It remains to prove that the resulting new curve $C$ is still a convex one. Note that it certainly consists of convex arcs $C_j$, so convexity of $C$ is equivalent to the statement that at every point $V_j$ of joining of the pieces, we still have a convex angle, i.e. the incoming tangent angular direction is exceeded by the outgoing angular direction.

First, if $V_j$ is the joining point of two \emph{curved} arcs, then neither the incoming $\Gamma_{j-1}$, nor the outgoing $\Gamma_j$
 is changed, whence original convexity of $\Gamma=\partial K$ ensures that there exists a (locally)
supporting line -- say $V_j+e^{i\alpha_{-}(v_j)}\RR$ -- to $C$ at $V_j$. So let us consider the cases when at $V_j$ there is some straight piece coming in or going out, and when, therefore, there is a jump of the tangent direction, i.e. an outer angle $\Omega(V_j) \ge \lambda(j) \xi$ according to Condition (4) in Definition \ref{def:Edomain} of an $E$-domain.

Now the incoming tangent direction is either $\alpha_{-}(V_j)$ (if the piece $\Gamma_{j-1}$ was a curved one and is thus not changed,
so that $C_{j-1}=\Gamma_{j-1}$) or $\alpha_{-}(V_j)+\xi$ (if $\Gamma_{j-1}$ was a straight line segment and is thus replaced by the
respective circular arc $\Lambda_{j-1}$). Similarly,
the outgoing angle is either $\alpha_{+}(V_j)$ (if $\Gamma_j$ is curved and $C_j=\Gamma_j$) or $\alpha_{+}(V_j)-\xi$
(if $\Gamma_j$ was a straight line piece and $C_j=\Lambda_{j}$ the constructed circular arc).
That is, for $C$ the difference of the outgoing angle and the incoming angle is exactly
$\alpha_{+}(V_j)-\alpha_{-}(V_j)-\lambda(j) \xi = \Omega(V_j)- \lambda(j) \xi \ge 0$ by assumption.
 Therefore, the curve $C$ has a (locally) supporting line --- e.g. $V_j+e^{i\alpha_{-}(V_j)+\xi}$ --- even at the join points $V_j$ where there is some straight line piece coming in or going out.

In all, there is a (locally) supporting line to $C$ at all vertices $V_j$ for all $j=1,\dots,k$. Therefore, in view of the convexity of all the arcs $C_j$ ($j=1,\dots,k$), $C$ is convex, too.

Moreover, for the $j$ with $C_j=\Gamma_j$, we already have the linearly a.e. condition that the curvature is at least $\kappa$, and for the newly constructed circular arc pieces $C_\ell=\Lambda_{\ell}$ we also have that the curvature is $1/R_\ell \ge 2\sin \xi /\Delta$ by \eqref{Rell}. That is, we find that the curvature of $C$ is at least $\kappa^{\star}:=\min(\kappa, 2\sin \xi /\Delta) >0$ linearly a.e., and Strantzen's Lemma \ref{l:roughcurvature} applies. Whence the assertion.
\end{proof}

\section{Calculation on the straight line segment boundaries}\label{ss:straight}

Let now $K$ be an $E(k,d,\Delta,\kappa,\xi,\delta)$-domain.
This also means that all the straight line segment boundary parts have length $L_j \le \Delta-\delta <\Delta$ each. Let one boundary arc $\Gamma=\Gamma_j$, which is a straight line segment, be fixed. Assume, as we may, that $\Gamma=[-a,a]$ with $L_j=2a\le \Delta-\delta$. Also we may assume that $K\subset \HP :=\{ z \in \CC~:~ \Im z \ge 0\}$. Now by condition at the endpoints of $\Gamma$ there is a jump of the tangent, and $\alpha_{-}(-a) \le -\xi$, $\alpha_{+}(a)\ge \xi$.

Since $K\subset \HP$ and the supporting lines $\pm a + e^{i \alpha_{\pm}(\pm a)} \RR$ at $\pm a$ have angles with $\RR$ at least $\xi$, we have $K\subset K'$, where $K'\subset \HP$ is the domain in the upper halfplane bounded by the halfline $s:=-a+e^{i(\pi-\xi)}\RR_{+}$, the segment $\Gamma=[-a,a]$, and the other halfline $t:=a+e^{i\xi}\RR_{+}$.

Let now $0<\theta <\xi/2$ be a small angle, and consider the ray (halfline) $\ell$, emanating from $-a$ in the direction of $e^{i\theta}$.
 Write $b:=|u+iv-a|$ and $c:=[u+iv+a|$. Obviously, then the point of intersection $u+iv:=\ell \cap t$ satisfies ${u=a+b\cos\xi=c\cos\theta-a}$ and ${v=b\sin \xi= c\sin \theta}$. Without any further trigonometrical calculus, it is clear that with $\theta \searrow 0$ we will have $b\to 0$, $u\to a$, $c\to 2a$, $v\to 0$ and $(u+iv) \to a$.

Drawing the halfline $m:=a+\RR_{+}e^{i(\pi-\theta)}$, by symmetry we will find $m \cap s = -u+iv $. So if we define the quadrangle\footnote{Actually, $B$ is a trapezoid. In view of $0<\xi\le \pi/2$, it is also clear that $\diam B =\max(2u,c)$.}
$$
B:=B(\theta):=\con\{(-u+iv),-a,a,(u+iv)\},
$$
then we will have $\diam B  <\Delta-\delta/2$ if $\theta \le \theta_0(\Delta,\xi,\delta)$ is small enough.

For an explicit constant here let us choose $\theta_0(\Delta,\xi,\delta):=\dfrac{\de \xi}{4\Delta}$, say.

The conditions $0<\tht<\xi \le \pi/2$ entail $0< \sin (\xi-\tht) < \sin \xi$ and applying the sine theorem both in the triangles with vertices $-a, a, u+iv$ and $-a, u+iv, -u+iv$, we find
$$
\frac{c}{2a}= \frac{\sin(\pi-\xi)}{ \sin(\xi-\theta)}=\frac{\sin \xi }{\sin(\xi-\theta)}
\quad\textrm{and} \quad \frac{2u}{c}= \frac{\sin(\pi-\xi-\tht)}{\sin \xi}=\frac{\sin (\xi+\tht) }{\sin\xi},
$$
respectively, whence
\begin{align*}
\diam B = \max(c,2u) & = 2a \frac{\sin \xi}{\sin(\xi-\theta)} \max\left(1, \frac{\sin (\xi+\tht)}{\sin\xi} \right)
\\ & = 2a \max\left(\frac{\sin \xi}{\sin(\xi-\theta)}, \frac{\sin (\xi+\tht)}{\sin(\xi-\tht)} \right) <(\Delta-\de) \frac{\sin \xi + \tht}{\sin \xi - \tht}.
\end{align*}

This last estimate will be below $\Delta-\de/2$ if (and only if) $\dfrac{\Delta-\de}{\Delta-\de/2} < \dfrac{\sin\xi-\theta}{\sin \xi + \tht},$
i.e. when $\dfrac{\de}{2\Delta-\de} > \dfrac{2\theta}{\sin \xi + \tht}$ or, equivalently, if $\sin \xi  > \dfrac{4\Delta-3\de}{\de}\theta $.

However, in view of $\sin \xi \ge 2 \xi/\pi$ ($0<\xi\le\pi/2$) and $\de \le \Delta/2$, our choice of $\theta_0$ ensures
$$
\theta\le\theta_0 = \frac{\de }{4\Delta} ~\xi \le \frac{\de}{4\Delta}  (\frac{\pi}{2} \sin \xi) < \frac{\de \sin \xi}{2.5 \Delta}
\le \frac{\de \sin \xi}{4\Delta-3\de},
$$
which suffices.

This is useful for the following. Put $S[\phi, \psi]:=\{ z \in \CC~:~ \phi\le \arg z \le \psi \}$ for the sector of angles between $\phi$ and $\psi$. If $z \in \Gamma$, then $z+S[0,\theta] \subset (-a+S[0,\theta])$,
and so $K\cap (z+S[0,\theta]) \subset K \cap (-a+S[0,\theta])$, and symmetrically
$K\cap (z+S[\pi-\theta,\pi]) \subset K \cap (a+S[\pi-\theta,\pi])$, whence both sets are contained in $B=B(\theta)$ and we obtain
$$
K(z,\theta):=K\cap \{(z+S[0,\theta]) \cup (z+S[\pi-\theta,\pi])\} \subset B.
$$
It follows that for any $\theta <\theta_0(\Delta,\xi,\delta) (<\xi/2)$ we have $K(z,\theta) \subset B(\theta)$ and $\diam K(z,\theta) \le \diam B < \Delta-\delta/2$.

\begin{lemma}\label{l:straightalternative} Let $\Gamma_j\subset \DK$ be a straight line
boundary piece of the $E$-domain $K$, and $0<\theta := \dfrac{\de \xi}{2\pi\Delta}
\left( < \theta_0(\Delta,\xi,\delta):=\dfrac{\de \xi}{4\Delta}\right)$. Then with $0<\eta :=\eta_0(\Delta,\delta):=\dfrac{\delta}{8\Delta}$ and for any $n\in \NN$ and any $p\in\PK$ we have the following alternative.
\begin{enumerate}[(i)]
\item Either for all $z\in \Gamma_j$ we have $|p'(z)| > \eta \dfrac{\sin \theta}{d} n |p(z)|$;
\item or for all $z\in \Gamma_j$ we have $|p(z)| \le \exp\left(-2\eta n \right) \|p\|_K$.
\end{enumerate}
\end{lemma}
\begin{proof} As above, put $\Gamma:=\Gamma_j$ and $\Gamma=[-a,a]$, $K \subset \HP$. Assume that (i) fails, so that there exists some $z\in \Gamma$ with
\begin{equation}\label{exceptionalz}
|p'(z)| \le \eta \dfrac{\sin \theta}{d} n |p(z)|.
\end{equation}
Let us write the zeros of $p$ in the form $z_j=z+\rj$ ($j=1,\dots,n$). Since ${p\in \PK}$ and $K\subset \HP$,
we obviously have $0\le \varphi_j\le \pi$ ($j=1,\dots,n$).
The full zero set $\Z$ splits into the subsets
$$\ZS:=\Z \cap K(z,\theta) = \Z \cap \{z+(S[0,\theta] \cup S[\pi-\theta,\pi])\} \quad \textrm{and} \quad  \W:=\Z\setminus \ZS.
$$
 For this latter subset of zeroes we will adopt Tur\'an's direct argument, to obtain at $z$
\begin{align*}\label{Turanargumentatz} \notag
\left|\frac{p'}{p}(z)\right| & \ge \Im \frac{p'}{p}(z) = \Im \sum_{j=1}^n \frac{1}{z-z_j} = \Im \sum_{j=1}^n  \frac{- 1}{r_j} e^{-i\varphi_j} = \sum_{j=1}^n \sj
\\ &\ge \sum_{z_j\in \W} \sj \ge \sum_{z_j\in \W} \frac{\sin \theta}{d} = \frac{\sin \theta}{d} \# \W.
\end{align*}
Comparing this to 
\eqref{exceptionalz} yields $\# \W \le \eta n$ in this case, whence we also have $\# \ZS \ge (1-\eta)n$.
Since $K(z,\theta)\subset B$, we also have $\ZS \subset B$. As $\|p\|_K\ge \Delta^n$ in view of Lemma \ref{l:FeketeSzego},
 this and the basic estimate that $\Delta_K \ge \diam K/4 =d/4$ (see \cite{RansSur}) implies for any $\zeta \in B$
\begin{equation*}
\label{transfiniteBK}
\begin{aligned}
\frac{|p(\zeta)|}{\|p\|_K}&=\frac{\prod_{j=1}^{n} |\ze-z_j| }{\|p\|_K}
\le
\frac{  \prod_{z_j\in \ZS} |\ze-z_j|\prod_{z_j\in \W} |\ze-z_j|}{\Delta^{n}}
\\ & \le \left(\frac{\diam B}{\Delta}\right)^{\# \ZS} \left(\frac{d}{\Delta}\right)^{\# \W}
\le \left(\frac{\Delta-\delta/2}{\Delta}\right)^{n-\# \W} 4^{\# \W}
\\ & \le \exp\bigg( n (1-\eta) \log\left(\frac{\Delta-\de/2}{\Delta}\right) +\eta ~n \log 4 \bigg) \\ 
& \le \exp \bigg( n \bigg\{ (1-\eta)\left(-\frac{\delta}{2\Delta}\right)+ \eta\log 4 \bigg\}\bigg),
\end{aligned}
\end{equation*}
using in the last step $\log (1-x) \le -x $ ($0<x<1$).
Observe that here within the curly brackets $\{ \}$ the constant depends continuously on $\eta$ and becomes negative for $\eta=0$, whence even for small $\eta$ it is already negative. To be more explicit, using $\log 4 \approx 1.386294361.. <1.5$ and $\delta < \Delta$ (whence $\delta/(2\Delta) <0.5$), we obtain
\begin{equation*}\label{pzetaconclusion}
\frac{|p(\zeta)|}{\|p\|_K} \le \exp \bigg( n \bigg\{ 2 \eta -\frac{\delta}{2\Delta} \bigg\}\bigg) \le \exp\left(-\frac{\delta}{4\Delta} n \right)\le e^{-2\eta n} \qquad \bigg( \eta := \eta_0:= \frac{\delta}{8\Delta} \bigg).
\end{equation*}
\end{proof}

\section{Completion of the proof of Theorem \ref{th:E}}\label{s:Edomainproof}

We fix the parameter values $\tht:=\dfrac{\de \xi}{2\pi\Delta}
\left( < \theta_0(\Delta,\xi,\delta):=\dfrac{\de \xi}{4\Delta}\right)$ and  $\eta:=\eta_0:=\dfrac{\de}{8 \Delta}$ as above.
Then we consider three subsets of $\DK$ taking\footnote{Here, and throughout this section, $(i)$ and $(ii)$
refer to the respective properties in Lemma \ref{l:straightalternative}.}
$$
\CR:=\mathop{\cup}_{\Gamma_j ~\textrm{is curved}} \Gamma_j, \quad
\LL:=\mathop{\cup}_{\Gamma_j ~\textrm{is straight} ~\& ~(i)~ \textrm{holds}} \Gamma_j,\quad \SC:= \mathop{\cup}_{\Gamma_j ~\textrm{is straight} ~\& ~(i)~ \textrm{fails}} \Gamma_j.
$$

By partial $R$-circularity, provided by Lemma \ref{l:partialRcircular} (with $R=R_K$ of the Lemma), on $\CR$ we can apply Tur\'an's Lemma \ref{Tlemma} at each point to get
$$
\int_{\CR} |p'|^q \ge \left( \frac{1}{2R}\right)^q n^q \int_{\CR} |p|^q .
$$
On the straight part $\LL$, where (i) holds, the situation is even simpler as (i) directly entails
$$
\int_{\LL} |p'|^q \ge \left( \frac{\delta \sin \theta}{8 \Delta d}\right)^q n^q \int_{\LL} |p|^q .
$$
Adding these and estimating trivially on the rest, we are led to
\begin{equation}\label{compareonCandL}
\|p'\|_q^q \ge \int_{\CR\cup\LL} |p'|^q \ge  \left(\min\left(\frac{1}{2R},~ \frac{\delta \sin \theta}{8 \Delta d}\right)\right)^q n^q \int_{\CR\cup\LL} |p|^q.
\end{equation}
Finally, on the straight part $\SC$, where (i) fails, Lemma \ref{l:straightalternative} guarantees (ii), whence here we have
$$
\int_{\SC} |p|^q \le |\SC| \exp\left(-nq \frac{\delta}{4\Delta} \right) \|p\|_K^q \le 2\pi d \exp\left(-nq \frac{\delta}{4\Delta} \right) \|p\|_K^q ,
$$
for $|\SC|\le L \le 2\pi d$ by convexity\footnote{A reference is \cite[p. 52, Property 5]{BF} about surface area, presented as a consequence of the Cauchy Formula for surface area.}. This we are to combine with Lemma \ref{l:Nikolskii}, more precisely with $\|p\|_K^q \le (2(q+1)/d) n^2 \|p\|_q^q$, directly following from \eqref{eq:Nikolskii}. This leads to
\begin{align*}
\int_{\SC} |p|^q &\le 2\pi d \exp\left(-nq \frac{\delta}{4\Delta} \right) \frac{2(q+1)}{d} n^2 \|p\|_q^q \\
&= \exp\left( \log(4\pi (q+1))+2\log n - n \frac{q \delta}{4\Delta} \right) \|p\|_q^q.
\end{align*}
Acting a little wasteful, with
say $n_0:=100 \left(\frac{\Delta}{\delta}\right)^2 \ge 100$ we surely have for all $n\ge n_0$ the estimate
\begin{align*}
&\log(4\pi(q+1))+2\log n - n \frac{q \delta}{4\Delta}
< -\log 2 +\log(8\pi)+q + \sqrt{n} - 2.5 q \sqrt{n}
\\ &< -\log 2 + 4 + q - 1.5 q \sqrt{n} < - \log 2 + 4 + q -15 q < - \log 2,
\end{align*}
since for $n\ge n_0 
$ we have $2\log n < \sqrt{n}$ and $10 \dfrac{\Delta}{\de} = \sqrt{n_0} \le \sqrt{n}$. So for $n\ge n_0$ we find
$$
\int_{\SC} |p|^q \le \exp\left( -\log 2 \right) \|p\|_q^q = \frac{1}{2} \|p\|_q^q.
$$
From here \eqref{compareonCandL} leads to
\begin{align*}
\|p'\|_q^q &\ge \left(\min\left(\frac{1}{2R},~ \frac{\delta \sin \theta}{8 \Delta d}\right)\right)^q n^q \left( \|p\|_q^q - \int_{\SC} |p|^q \right)
\\& \ge \frac12 \left(\min\left(\frac{1}{2R},~ \frac{\delta \sin \theta}{8 \Delta d}\right)\right)^q ~ n^q ~ \|p\|_q^q \qquad \left( n \ge n_0:=100 \left(\frac{\Delta}{\delta}\right)^2\right)
\end{align*}
hence
\begin{equation}\label{finallargen}
\|p'\|_q \ge \frac12 \min\left(\frac{1}{2R},~ \frac{\delta \sin \theta}{8 \Delta d}\right) ~ n~ \|p\|_q
\qquad \left( n \ge n_0:=100 \left(\frac{\Delta}{\delta}\right)^2\right).
\end{equation}
Substituting the expressions for $R=R_K$ from Lemma \ref{l:partialRcircular} and $\theta$ from Lemma \ref{l:straightalternative}
and recalling that $\delta \le \Delta/2$, $\xi\in (0,\pi)$ and $\de < \Delta \le d$, we see that
\begin{align*}
\min\left(\frac{1}{2R},~ \frac{\delta \sin \theta}{8 \Delta d}\right)& =
\min\left(\frac{\kappa}{2},~\frac{\sin \xi }{\Delta},~ \frac{\delta \sin \left(\frac{\delta \xi}{2\pi\Delta}\right)}{8 \Delta d}\right)=
\min\left(\frac{\kappa}{2},~\frac{\delta \sin \left(\frac{\delta \xi}{2\pi\Delta}\right)}{8 \Delta d}\right)
\\& \ge \min\left(\frac{\kappa}{2},~0.019 \left(\frac{\delta}{\Delta}\right)^2\frac{\xi}{ d}\right),
\end{align*}
because $\dfrac{\sin t}{t}$ decreases in $[0,1/4]$ and therefore for $\dfrac{\delta \xi}{2\pi\Delta} <1/4$ we can write
$$
\sin \left(\dfrac{\delta \xi}{2\pi\Delta}\right) > \dfrac{\sin(1/4)}{1/4} \cdot \dfrac{\delta \xi}{2\pi\Delta} > 0.989 \cdot \dfrac{\delta \xi}{2\pi\Delta}
$$ and thus
$$\displaystyle \frac{\delta \sin \left(\frac{\delta \xi}{2\pi\Delta}\right)}{8 \Delta d} \ge
\dfrac{\de ~ 0.989 \cdot \frac{\delta \xi}{2\pi\Delta}}{8 \Delta d} > 0.019 \left(\frac{\delta}{\Delta}\right)^2\frac{\xi}{d}.
$$
 Applying this in \eqref{finallargen} leads to $\|p'\|_q \ge \min\left(\dfrac{\kappa}{4}, 0.009 \left(\dfrac{\delta}{\Delta}\right)^2 \dfrac{\xi}{d}\right) ~n ~\|p\|_{q}$ for $n\ge n_0$.

Finally, for any convex domain $K$ and any $n\le n_0$, the above Lemma \ref{l:Gabriel}
yields $\|p'\|_q \ge 0.022 \dfrac{1}{d} \|p\|_q \ge 0.022 \dfrac{1}{d n_0} n \|p\|_q $. This furnishes for $n \le n_0$ the estimate $\displaystyle \|p'\|_q \ge 0.022 \dfrac{1}{100 \left(\frac{\Delta}{\delta}\right)^2 d} n \|p\|_q > 0.00022 \left(\frac{\delta}{\Delta}\right)^2 \dfrac{1}{d} n \|p\|_q $.

In all, we get $\|p'\|_q > c_K \|p\|_q$ with
$\displaystyle c_K:=\min\left(\dfrac{\kappa}{4}, 0.00022 \left(\frac{\delta}{\Delta}\right)^2 \dfrac{1}{d}, 0.009 \left(\dfrac{\delta}{\Delta}\right)^2 \dfrac{\xi}{d}\right)$.


Whence the assertion.



\end{document}